\documentclass[preprint]{elsarticle}

\usepackage{amssymb}
\usepackage{amsmath}
\usepackage{amsthm}
\usepackage{enumerate}
\usepackage{graphicx}
\usepackage{amsfonts}
\usepackage{mathtools,xcolor}

\journal{...}

\newtheorem{theorem}{Theorem}[section]

\newtheorem{lemma}{Lemma}[section]

\numberwithin{equation}{section}

\begin{document}

	\begin{frontmatter}

		\title{An asymptotic formula for Aldaz-Kounchev-Render operators on the hypercube}

		\author[1]{Ana-Maria Acu} 
			\author[2]{Ioan Ra\c sa}
		
		\vspace{10cm}
		

		\address[1]{Lucian Blaga University of Sibiu, Department of Mathematics and Informatics,  Romania, e-mail: anamaria.acu@ulbsibiu.ro}

		\address[2]{Technical University of Cluj-Napoca, Faculty of Automation and Computer Science, Department of Mathematics, Str. Memorandumului nr. 28, 400114 Cluj-Napoca, Romania
			e-mail:  ioan.rasa@math.utcluj.ro }

		\begin{abstract} 	
			{ We prove a version of a conjecture concerning the asymptotic behavior of the Aldaz-Kounchev-Render operators on the hypercube.} 
		\end{abstract}

		
	\end{frontmatter}

	\section{Introduction}
	
		Let   $B_{n}^{[1]}:C[0,1]\to C[0,1]$ be the classical Bernstein operator defined as
	$$B_n^{[1]}f(x)=\displaystyle\sum_{i=0}^n f\left(\dfrac{i}{n}\right)p_{n,i}(x),$$ where   $p_{n,i}(x)={n\choose i}x^i(1-x)^{n-i},\, x\in [0,1]$.
	
		For a fixed $j\in {\mathbb N}$, $j\geq 2$ and for  $n\geq j$,  Aldaz, Kounchev and Render \cite{AKR2009} introduced a { polynomial}  operator $B_{n,j}^{[1]}: C[0,1]\to C[0,1]$ that fixes $e_0$ and $e_j$, investigated its approximation properties and gave applications to CAGD.  The operator  is explicitly given by
	\begin{equation*}
		B_{n,j}^{[1]}f(x)=\displaystyle\sum_{k=0}^n f\left(t_{n,k}^j\right)p_{n,k}(x),\end{equation*}
	where $$t_{n,k}^j=\left(\dfrac{k(k-1)\dots(k-j+1)}{n(n-1)\dots(n-j+1)}\right)^{1/j}.$$
	The Voronovskaja-type formula for the sequence $(B_{n,j}^{[1]})_{n\geq 1}$ was conjectured in \cite{Rasa2012} and proved in \cite{Bir2017}, \cite{Gavrea}.
	
	For $f\in C([0,1]^2)$,  the tensor product  $B_{n}^{[1]}\otimes B_{n}^{[1]}$ is given by
	\begin{equation}\label{Bernstein} B_{n}^{[2]}
		f(x,y):=(B_{n}^{[1]}\otimes B_{n}^{[1]})f(x,y)=\displaystyle\sum_{k=0}^{n}\sum_{l=0}^{n}f\left(\dfrac{k}{n}, \frac{l}{n}\right)p_{n,k}(x) p_{n,l}(y). \end{equation}
		Let   $B_{n,j}^{[1]}:C[0,1]\to C[0,1]$ be the AKR operator  and $(x,y)\in [0,1]^2$. Then, for $f\in C([0,1]^2)$,  the tensor product $B_{n,j}^{[1]}\otimes B_{n,j}^{[1]}$  is given by
	\begin{align}\label{Aldaz}B_{n,j}^{[2]}f(x,y)&:=(B_{n,j}^{[1]}\otimes B_{n,j}^{[1]})f(x,y)\nonumber\\
		&=\displaystyle\sum_{k=0}^{n}\sum_{l=0}^{n}f\left(t_{n,k}^j,t_{n,l}^j\right)p_{n,k}(x) p_{n,l}(y),  \,\, (x,y)\in [0,1]^2.\end{align}
	
	A conjecture concerning the Voronovskaja-type formula for the sequence $(B_{n,j}^{[2]})$ was formulated in  \cite{RIMA}. The aim of this paper is to prove a version of this conjecture.

	\section{Proof of Conjecture} For the sake of conciseness we consider only the case $j=2$, but obviously the proof can be extended to arbitrary $j$.
	
	Let $k$ and $n$ be integers, $n\geq 2$, $0\leq k\leq n$. Define
	$$R(n,k):=\dfrac{k}{n}-\sqrt{\dfrac{k(k-1)}{n(n-1)}}-\dfrac{1}{2n}+\dfrac{k}{2n^2}.  $$
	It is elementary to prove that
	\begin{align}
	&R(n,0)=-\dfrac{1}{2n},\label{eq.2.1}\\
	&R(n,k)\geq 0,\,\,k=1,2,\dots,n,\label{eq.2.2}\\
	& 0\leq \dfrac{k}{n}-\sqrt{\dfrac{k(k-1)}{n(n-1)}}\leq \dfrac{1}{n}.\label{eq.2.3}
	\end{align}
\begin{lemma} If $0<x\leq 1$, then
	\begin{equation}
		\displaystyle\lim_{n\to\infty}n\sum_{k=1}^n p_{n,k}(x)R(n,k)=0.\label{eq.2.4}
	\end{equation}
	\end{lemma}
\begin{proof} Let $x\in (0,1]$ and $f\in C^2[0,1]$. It is known (see \cite{Bir2017}, \cite{Gavrea}) that
	$$ \displaystyle\lim_{n\to\infty} n(B_{n,2}^{[1]}f(x)-f(x))=\dfrac{x(1-x)}{2}f^{\prime\prime}(x)-\dfrac{1-x}{2}f^{\prime}(x).  $$
	It is also well known that
	$$ \displaystyle\lim_{n\to\infty} n(B_n^{[1]}f(x)-f(x))=\dfrac{x(1-x)}{2}f^{\prime\prime}(x). $$
	It follows that
	$$\displaystyle\lim_{n\to\infty}n\left(B_{n,2}^{[1]}f(x)-B_n^{[1]}f(x)\right)=-\dfrac{1-x}{2}f^{\prime}(x).  $$
	In particular, for the function $f(t)=t$, we get
	$$ \displaystyle\lim_{n\to\infty}n\displaystyle\sum_{k=1}^n p_{n,k}(x)\left(\sqrt{\dfrac{k(k-1)}{n(n-1)}}-\dfrac{k}{n}\right) =-\dfrac{1-x}{2}. $$
	This can be written as
	$$ \displaystyle\lim_{n\to\infty}n\sum_{k=1}^n p_{n,k}(x)\left[\dfrac{1}{2n}\left(1-\dfrac{k}{n}\right)+R(n,k)\right]=\dfrac{1-x}{2},  $$
	i.e.,
	\begin{equation}
		\label{eq.2.5}\dfrac{1}{2}\lim_{n\to\infty}\sum_{k=1}^n p_{n,k}(x)\left(1-\dfrac{k}{n}\right)+\lim_{n\to\infty} n\sum_{k=1}^n p_{n,k}(x)R(n,k)=\dfrac{1-x}{2}.
	\end{equation}
Let us remark that
$$ \dfrac{1}{2}\lim_{n\to\infty}\sum_{k=1}^n p_{n,k}(x)\left(1-\dfrac{k}{n}\right)=\dfrac{1}{2}\lim_{n\to\infty}\left(B_n^{[1]}(1-t;x)-(1-x)^n\right)=\dfrac{1}{2}(1-x). $$
Combined with (\ref{eq.2.5}) this leads to (\ref{eq.2.4}), and the proof is finished.
	\end{proof}
\begin{theorem}
	Let $0<x\leq 1$, $0<y\leq 1$, $f\in C^2([0,1]^2)$. Then
\begin{align}
&	\displaystyle\lim_{n\to\infty} n\left(B_{n,2}^{[2]}f(x,y)-f(x,y)\right)\nonumber\\
	&=\dfrac{x(1-x)}{2}f_{x^2}^{\prime\prime}(x,y)+\dfrac{y(1-y)}{2}f_{y^2}^{\prime\prime}(x,y)-\dfrac{1-x}{2}f_x^{\prime}(x,y)-\dfrac{1-y}{2}f_y^{\prime}(x,y).\label{eq.2.6}
\end{align}
\end{theorem}
\begin{proof} First we have
	\begin{align*}
&n\left(B_{n,2}^{[2]}f(x,y)-B_n^{[2]}f(x,y)\right)\\&=n\displaystyle\sum_{k=0}^n\sum_{l=0}^np_{n,k}(x)p_{n,l}(y)
\left[f\left(\sqrt{\dfrac{k(k-1)}{n(n-1)}},\sqrt{\dfrac{l(l-1)}{n(n-1)}}\right)-f\left(\dfrac{k}{n},\dfrac{l}{n}\right)\right]\\
&=E_nf(x,y)+F_nf(x,y)+G_nf(x,y),
	\end{align*}
where
\begin{align*}
E_nf(x,y)&:=n\sum_{k=0}^n\sum_{l=0}^n p_{n,k}(x)p_{n,l}(y)\left(\sqrt{\dfrac{k(k-1)}{n(n-1)}}-\dfrac{k}{n}\right)f^{\prime}_x\left(\dfrac{k}{n},\dfrac{l}{n}\right),\\
F_nf(x,y)&:=n\sum_{k=0}^n\sum_{l=0}^n p_{n,k}(x)p_{n,l}(y)\left(\sqrt{\dfrac{l(l-1)}{n(n-1)}}-\dfrac{l}{n}\right)f^{\prime}_y\left(\dfrac{k}{n},\dfrac{l}{n}\right),\\
G_nf(x,y)&:=\dfrac{n}{2}\sum_{k=0}^n\sum_{l=0}^n p_{n,k}(x)p_{n,l}(y)\left\{\left(\sqrt{\dfrac{k(k-1)}{n(n-1)}}-\dfrac{k}{n}\right)^2f^{\prime\prime}_{x^2}(\xi,\eta)\right.\\
& +2\left(\sqrt{\dfrac{k(k-1)}{n(n-1)}}-\dfrac{k}{n}\right)\left(\sqrt{\dfrac{l(l-1)}{n(n-1)}}-\dfrac{l}{n}\right)f^{\prime\prime}_{xy}(\xi,\eta)\\
&+\left.\left(\sqrt{\dfrac{l(l-1)}{n(n-1)}}-\dfrac{l}{n}\right)^2f^{\prime\prime}_{y^2}(\xi,\eta)\right\},
\end{align*}
for suitable $(\xi,\eta)$ furnished by Taylor's formula. Using (\ref{eq.2.3}) we see that
\begin{equation}
	\label{eq.2.7}
	\displaystyle\lim_{n\to\infty}G_nf(x,y)=0.
\end{equation}
Moreover,
\begin{align*}
	&\displaystyle\lim_{n\to\infty} E_nf(x,y)\\
	&=-\displaystyle\lim_{n\to\infty} n\sum_{k=0}^n\sum_{l=0}^n p_{n,k}(x)p_{n,l}(y) \left[\dfrac{1}{2n}\left(1-\dfrac{k}{n}\right)+R(n,k)\right]f_x^{\prime}\left(\dfrac{k}{n},\dfrac{l}{n}\right)\\
	&=-\dfrac{1}{2}\displaystyle\lim_{n\to\infty}\sum_{k=0}^n\sum_{l=0}^n p_{n,k}(x)p_{n,l}(y)\left(1-\dfrac{k}{n}\right)f_x^{\prime}\left(\dfrac{k}{n},\dfrac{l}{n}\right)\\
	&-\displaystyle\lim_{n\to\infty}n\sum_{k=0}^n\sum_{l=0}^np_{n,k}(x)p_{n,l}(y)R(n,k)f_x^{\prime}\left(\dfrac{k}{n},\dfrac{l}{n}\right)\\
	&=-\dfrac{1}{2}\lim_{n\to\infty} B_n^{[2]}\left((1-s)f^{\prime}_x(s,t);(x,y)\right)\\
&-\displaystyle\lim_{n\to\infty}n\sum_{k=1}^n\sum_{l=0}^n p_{n,k}(x)p_{n,l}(y)R(n,k)f_x^{\prime}\left(\dfrac{k}{n},\dfrac{l}{n}\right)\\
&+\lim_{n\to\infty}n\sum_{l=0}^n (1-x)^np_{n,l}(y)\dfrac{1}{2n}f_x^{\prime}\left(0,\dfrac{l}{n}\right).
\end{align*}
The first term equals $-\dfrac{1}{2}(1-x)f_x^{\prime}(x,y)$.

Moreover, using (\ref{eq.2.2}) we have
\begin{align*}
	&\left|  \displaystyle n\sum_{k=1}^n\sum_{l=0}^np_{n,k}(x)p_{n,l}(y)R(n,k)f_x^{\prime}\left(\dfrac{k}{n},\dfrac{l}{n}\right) \right|\\
	&\leq\sum_{l=0}^n\left(n\sum_{k=1}^n p_{n,k}(x)R(n,k)
	\|f_x^{\prime}\|_{\infty}\right)p_{n,l}(y)\\
	&=n\sum_{k=1}^n p_{n,k}(x)R(n,k)\|f_x^{\prime}\|_{\infty},
\end{align*}
and (\ref{eq.2.4}) shows that the second term is zero. The third one is also zero, and so
\begin{equation}
	\label{eq.2.8}
	\displaystyle\lim_{n\to\infty} E_n f(x,y)=-\dfrac{1-x}{2}f_x^{\prime}(x,y).
\end{equation}
Similarly,
\begin{equation}
	\label{eq.2.9} 
	\lim_{n\to\infty} F_nf(x,y)=-\dfrac{1-y}{2}f_y^{\prime}(x,y).
\end{equation}
Now (\ref{eq.2.7}), (\ref{eq.2.8}), (\ref{eq.2.9}) yield
\begin{equation}
	\label{eq.2.10}
	\displaystyle\lim_{n\to\infty}n\left(B_{n,2}^{[2]}f(x,y)-B_n^{[2]}f(x,y)\right)=-\dfrac{1-x}{2}f_x^{\prime}(x,y)-\dfrac{1-y}{2}f_y^{\prime}(x,y).
\end{equation}
On the other hand, it is well known that
\begin{equation}
	\label{eq.2.11}
	\displaystyle\lim_{n\to\infty} n(B_n^{[2]}f(x,y)-f(x,y))=\dfrac{x(1-x)}{2}f_{x^2}^{\prime\prime}(x,y)+\dfrac{y(1-y)}{2}f_{y^2}^{\prime\prime}(x,y).
\end{equation}
From (\ref{eq.2.10}) and (\ref{eq.2.11}) we get (\ref{eq.2.6}) and the theorem is proved.
\end{proof}

\end{document}